\documentclass[12pt]{amsart}
\usepackage{amsmath,amsfonts,amssymb,mathrsfs,amstext,amscd,latexsym,amsthm}
\usepackage{comment}
\usepackage{mathtools}
\usepackage{hyperref}
\usepackage{enumitem}
\usepackage{xcolor}
\newcommand{\R}{\ensuremath{\mathbb{R}}}

\newcommand{\N}{\ensuremath{\mathbb{N}}}

\DeclareMathOperator{\inrad}{\mathrm{inrad}}

\DeclareMathOperator{\spt}{\mathrm{spt}}

\DeclareMathOperator{\tr}{\mathrm{tr}}

\newcommand{\pd}{\partial}
\newcommand{\cd}{\nabla}
\newcommand{\inner}[2]{\left\langle #1 \, , \, #2\right\rangle} %Euclidean inner product
\newcommand{\lb}{\left(}
\newcommand{\rb}{\right)}
\newcommand{\lsb}{\left[}
\newcommand{\rsb}{\right]}

\newcommand{\Ges}{G_{\varepsilon,\sigma}}
\newcommand{\Gesp}{G_{\varepsilon,\sigma,+}}

\def\labelitemi{--}
\def\ba #1\ea {\begin{align} #1\end{align}}
\def\bann #1\eann {\begin{align*} #1\end{align*}}
\def\ben #1\een {\begin{enumerate} #1\end{enumerate}}
\def\bi #1\ei {\begin{itemize}\renewcommand\labelitemi{--} #1\end{itemize}}
\theoremstyle{plain}
\newtheorem{thm}{Theorem}%[section]
\newtheorem*{thm*}{Theorem}

\newtheorem{cor}[thm]{Corollary}

\newtheorem{prop}[thm]{Proposition}

\newtheorem*{conds*}{Conditions}
\newtheorem*{auxconds*}{Ancillary Conditions}
\newtheorem*{props*}{Properties}
\theoremstyle{remark}
\newtheorem{rem}{Remark}%[section]

\usepackage{cite}

\title[Local convexity estimates for MCF]{Local convexity estimates for mean curvature flow}

\author{Mat Langford}
\date{\today}
\thanks{This research was partially supported by ARC DECRA grant DE200101834.}

\address{Department of Mathematics, University of Tennessee Knoxville, Knoxville TN, USA, 37996-1320}
\address{School of Mathematical and Physical Sciences, The University of Newcastle, Newcastle, NSW, Australia, 2308}
\email{mlangford@utk.edu, mathew.langford@newcastle.edu.au}

\begin{document}

\begin{abstract}
We develop a local version of Huisken--Stampacchia iteration, using it to obtain local versions of a host of important sharp curvature pinching estimates for mean curvature flow. %and obtain a corresponding local estimate for the gradient of the second fundamental form. 
The local estimates we obtain do not depend on the quality of noncollapsing of the solution and the method adopted applies in a host of other settings.% (indeed, they hold for immersed flows). %These estimates are used to prove that an $(m_0+1)$-convex immersed mean curvature flow in $\R^{n+1}$ with $m_0<\frac{2(n-1)}{3}$ necessarily forms an arbitrarily high quality $\R^m\times S^{n-m}$-type neck with $m\leq m_0$ just prior to a singularity.
\end{abstract}

\maketitle

%\tableofcontents

%\section{Introduction}

%A one-parameter family $\{M_t^n\}_{t\in I}$ of smoothly immersed hypersurfaces $M_t^n$ of $\R^{n+1}$ \emph{evolves by mean curvature flow} if there exists a smooth family $X:M^n\times I\to\R^{n+1}$ of smooth immersions $X(\cdot,t):M^n\to\R^{n+1}$ with $X(M,t)=M_t$ satisfying
%\bann
%\pd_tX=\vec H\,,
%\eann
%where $\vec H(\cdot,t)\doteqdot \Div(DX(\cdot,t))$ is the mean curvature vector of $X(\cdot,t)$. %If $J\subset I$ and $X_t\doteqdot X(\cdot,t)$ is a proper immersion into $U\osubset \R^{n+1}$ for each $t\in J$, we say that $X:M^n\times I\to\R^{n+1}$ (or $\{M_t\}_{t\in I}$) is \emph{properly defined in} $U\times J$.

%We will denote $Q_r\doteqdot Q(0,0)$ and $Q\doteqdot Q_1$ and sometimes conflate the sets $Q_r(Y,s)\subset \R^{n+1}\times \R$ and $\{(x,t)\in M\times I: (X(x,t),t)\in Q_r(Y,s)\}\subset M\times I$.

%The most prominent theme of research in mean curvature flow is the investigation of the structure of singularities which occur during the evolution. 
Developing a quantitative structure theory for singularities in geometric flows is a fundamental problem, since their emergence prevents the flow from reaching an equilibrium state, and hence obstructs many desired applications. One powerful tool for analyzing singularity formation in extrinsic geometric flows is Huisken--Stampacchia iteration, which is the main tool in the proof of various important curvature pinching estimates that improve at the onset of singularities. %\cite{ALM15,AL14,ALM14,Hu84,HuSi99a,HuSi99b,HuSi09,L17,LaLy,Lynch}. 
For mean curvature flow, these a priori estimates imply that a compact, strictly $(m+1)$-convex\footnote{Recall that a hypersurface of $\R^{n+1}$ is \emph{strictly $m$-convex} if the sum $\kappa_1+\dots+\kappa_m$ of its smallest $m$ of its principal curvatures $\kappa_1\le\kappa_2\le\dots\le\kappa_n$ is positive.} hypersurface evolving by mean curvature becomes weakly convex and either becomes strictly $m$-convex or forms a high quality $m$-neck in regions of sufficiently high curvature \cite{Hu84,HuSi99a,HuSi99b,HuSi09,L17} (see \cite{HK1} for a different approach for embedded flows). This provides a powerful description of singularity formation; however, there remains the major drawback that the estimates depend on \emph{global} data from the initial (compact) hypersurface, whereas singularity formation is a \emph{local} phenomenon. This prevents application of the estimates to noncompact solutions. It also prevents iteration of the estimates in a neighbourhood of a singularity, which seems to be a basic requirement for extending the Huisken--Sinestrari surgery algorithm for 2-convex hypersurfaces \cite{HuSi09} to weaker intermediate convexity conditions.

We will localize these pinching estimates by introducing suitable cutoff functions and developing a local version of the Huisken--Stampacchia method. %(following Stampacchia \cite{Stampacchia66}). Following \cite[Theorem 6.1]{HuSi09}, 
Fixing data $n\in \N\setminus\{1\}$, $m\in \{0,\dots,n-1\}$, $\alpha>0$, $\Theta<\infty$, $V<\infty$, $q>0$, $\lambda>0$ and $\delta\in(0,\lambda)$ such that $\sqrt{q}\,\Theta\ge (\lambda-\delta)^{-1}$ and $\Theta V\ge 1$, and a scale parameter $R>0$, our main result may be stated as follows.

\begin{thm}[Local convexity and cylindrical estimates]\label{thm:convexity_and_cylindrical_estimates2}
If a mean curvature flow is properly defined in $B_{\lambda R}\times (0,\frac{1}{2n}R^2)\subset \R^{n+1}\times \R$ and satisfies
\begin{itemize}
\item $\displaystyle\inf_{B_{\lambda R}\times\{0\}\cup \pd B_{\lambda R}\times (0,\frac{1}{2n}R^2)}\frac{\kappa_1+\dots+\kappa_{m+1}}{\vert A\vert}\ge \alpha>0$,
\item $\displaystyle\sup_{B_{\lambda R}\times\{0\}\cup B_{\lambda R}\setminus B_{(\lambda-\delta)R}\times (0,\frac{1}{2n}R^2)}\tfrac{1}{n}H\le \Theta R^{-1}$, and
\item $\frac{R^2}{2n}\displaystyle\int_{B_{\lambda R}}\!\!\!H^q\,d\mu_0+\delta^{-2}\int_0^{\frac{R^2}{2n}}\!\!\!\!\int_{B_{\lambda R}\setminus B_{(\lambda-\delta)R}}\!\!\!\!\!H^q\,d\mu_t\,dt\le (\Theta R^{-1})^q(VR)^{n+2}$,
\end{itemize}
then, given any $\varepsilon>0$ and $\vartheta\in(0,1)$, it satisfies the estimates
\ba\label{eq:convexity}
\kappa_1\geq -\varepsilon H-C_\varepsilon R^{-1}%\;\; \text{in}\;\; B_{(\lambda-\delta)\vartheta R}\times(-\tfrac{1}{2n}R^2,0)
\ea
and
\ba\label{eq:m convexity}
\sum_{j=m+1}^n(\kappa_n-\kappa_j)\le \sum_{j=1}^m\kappa_j+\varepsilon H+C_\varepsilon R^{-1}%\;\;\text{in}\;\; B_{(\lambda-\delta)\vartheta R}\times(-\tfrac{1}{2n}R^2,0)\,,
\ea
in $B_{(\lambda-\delta)\vartheta R}\times(0,\tfrac{1}{2n}R^2)$, where $C_\varepsilon=c(n,\alpha,q,\varepsilon)\Theta \left[\frac{\Theta V}{1-\vartheta}\right]^{\frac{2}{q}}$.
\end{thm}

\begin{rem}\mbox{}
\begin{itemize}
\item The hypotheses of the theorem are very close in nature to those of the aforementioned global estimates; however, since the estimates are local, appropriate conditions must also be assumed at the spatial boundary. Cf. \cite{Mramor}.
\item Taking $\lambda$ sufficiently large, we see that the (global) pinching estimates for mean curvature flow %of Huisken, Huisken--Sinestrari and others 
are an immediate consequence of Theorem \ref{thm:convexity_and_cylindrical_estimates2}. 
\item Local estimates of a similar nature have been obtained in \cite{HK1} (cf. \cite{Mramor}) assuming a noncollapsing condition (in the sense of \cite{An12,ShWa09}); however, the mean curvature flow in Theorem \ref{thm:convexity_and_cylindrical_estimates2} need not be noncollapsing, nor even embedded; it need only be \emph{proper}, which we take to mean that the pair $(X,t):M\times (t_{\mathrm{i}},t_{\mathrm{f}})\to \R^{n+1}\times \R$ forms a proper map with respect to the subset $B_{\lambda R}\times(-\frac{1}{2n}R^2,0)\subset \R^{n+1}\times\R$, where $X:M\times (t_{\mathrm{i}},t_{\mathrm{f}})\to \R^{n+1}$ is a parametrization for the flow and $t:M\times (t_{\mathrm{i}},t_{\mathrm{f}})\to \R$ is the projection onto the second (time) factor. Moreover, the method adopted here can be applied in other contexts, such as flows by nonlinear speeds, high codimension mean curvature flow, or free-boundary mean curvature flow.
\item The estimates can be localized in more general open subsets of $\R^{n+1}$, or parabolically open subsets of $\R^{n+1}\times\R$, by exploiting different cut-off functions.
\item Due to the pointwise curvature bound of the second hypothesis in Theorem \ref{thm:convexity_and_cylindrical_estimates2}, the integral curvature bound of the third hypothesis may be replaced by an area bound of the form
\[
\frac{R^2}{2n}\displaystyle\int_{B_{\lambda R}}\!\!\!\,d\mu_0+\delta^{-2}\int_0^{\frac{R^2}{2n}}\!\!\!\int_{B_{\lambda R}\setminus B_{(\lambda-\delta)R}}\!\!\!\!\!\,d\mu_t\,dt\le (VR)^{n+2}\,.
\] 
In fact, as can be easily deduced from the proof, each of these hypotheses need only be made in regions where $\varepsilon$-pinching fails.%the conclusion of the theorem does not hold.
\item Even in the global setting, Theorem \ref{thm:convexity_and_cylindrical_estimates2} provides a more precise accounting of the dependence of the constant $C_\varepsilon$ on the boundary data. %In particular, we see that certain rigidity results for the shrinking sphere or shrinking cylinder amongst convex ancient solutions (see \cite{HaHe,HuSi15}) may be recovered from \eqref{eq:m convexity}. See Corollary \ref{cor:ancient rigidity} below.
\item By applying maximum principle type arguments similar to those of \cite[Theorems 6.1 and 6.3]{HuSi09}, Theorem \ref{thm:convexity_and_cylindrical_estimates2} yields corresponding local derivative estimates for the second fundamental form $A$, so long as $m<\frac{2(n-1)}{3}$, and therefore also a corresponding local neck detection lemma (cf. \cite[Lemma 7.4]{HuSi09}). This makes a local surgery algorithm possible when $m=1$ (and $n\ge 3$).
\end{itemize}
\end{rem}

\begin{proof}[Proof of Theorem \ref{thm:convexity_and_cylindrical_estimates2}]
Let $G$ be given either by 
\[
G\doteqdot -\kappa_1\;\;\text{or by}\;\; G\doteqdot \kappa_n-\tfrac{1}{n-m}H
\]
and set, for any $\varepsilon>0$ and $\sigma\in (0,1)$,
\bann
G_\varepsilon\doteqdot G-\varepsilon(H-\tfrac{\alpha}{2}|A|)\,,\;\; \Ges\doteqdot G_\varepsilon H^{\sigma-1}\;\; \text{and}\;\; \Gesp\doteqdot \max\{\Ges,0\}\,.
\eann
Well known calculations then show that
\ba\label{eq:evolveGes}
\frac{(\pd_t-\Delta)\Ges}{\Ges}\leq \sigma|A|^2-\gamma\frac{|\cd A|^2}{H^2}+\gamma^{-1}\frac{|\cd\Ges|^2}{\Ges^2}
\ea
in $B_\lambda\times (0,\frac{1}{2n})\cap\,\spt \Ges$ in the distributional sense\footnote{In what follows, all differential inequalities are intended in the distributional sense.}, where $\gamma=\gamma(n,\alpha,\varepsilon)>0$ (see, for example, \cite[Proposition 12.9]{EGF} and \cite[Section 3]{L17}). 

Given any $\zeta\in C^\infty_0(B_\lambda)$, set $\psi\doteqdot \zeta\circ X$. Setting $v\doteqdot \Gesp^\frac{p}{2}$ we then obtain
\bann
\frac{(\pd_t-\Delta)\psi^2v^2}{\psi^2v^2}%={}&\frac{(\pd_t-\Delta)\psi^2}{\psi^2}+\frac{(\pd_t-\Delta)\Ges^p}{\Ges^p}-2\inner{\frac{\cd\psi^2}{\psi^2}}{\frac{\cd\Ges^p}{\Ges^p}}\\
%={}&2\frac{(\pd_t-\Delta)\psi}{\psi}-2\frac{|\cd\psi|^2}{\psi^2}+\frac{p(\pd_t-\Delta)\Ges}{\Ges}\\
%{}&-p(p-1)\frac{|\cd\Ges|^2}{\Ges^2}-4p\inner{\frac{\cd\psi}{\psi}}{\frac{\cd\Ges}{\Ges}}\\
={}&2\frac{(\pd_t-\Delta)\psi}{\psi}-2\frac{|\cd\psi|^2}{\psi^2}+\frac{p(\pd_t-\Delta)\Ges}{\Ges}\\
{}&-4\frac{p-1}{p}\frac{|\cd v|^2}{v^2}-8\inner{\frac{\cd\psi}{\psi}}{\frac{\cd v}{v}}
\eann
wherever $\psi v>0$. Applying Young's inequality to the final term and recalling \eqref{eq:evolveGes} we obtain, for $p\geq 6(1+\gamma^{-1})$,
\bann
\frac{(\pd_t-\Delta)\psi^2v^2}{\psi^2v^2}\leq{}&2\frac{(\pd_t-\Delta)\psi}{\psi}+6\frac{|\cd\psi|^2}{\psi^2}-\lb 2-\frac{4}{p}\rb\frac{|\cd v|^2}{v^2}\\
{}&+p\lb\sigma|A|^2-\gamma\frac{|\cd A|^2}{H^2}+\frac{4}{\gamma p^2}\frac{|\cd v|^2}{v^2}\rb\,.
\eann
Fix $r$ and $R$ so that $0<r<R\le \lambda$. If we choose the function $\zeta:\R^{n+1}\to\R$ so that
\ben
\item $\displaystyle \zeta(X)=0$ when $X\notin B_R$ and $\displaystyle \zeta(X)=1$ when $X\in B_r$,
\item $\displaystyle |D_i\zeta|^2\leq 10(R-r)^{-2}\zeta$ for each $i$, and
\item $\displaystyle |D_iD_j\zeta|\leq 10(R-r)^{-2}$ for each $i$ and $j$,
\een
then
\bann
(\pd_t-\Delta)\psi%={}&-\tr_{TM}D^2\zeta\circ X\nonumber\\
\leq{}&10n(R-r)^{-2}\chi_{B_R\setminus B_r}%\label{eq:evolvepsielliptic}
\;\;\text{and}\;\;
\frac{|\cd\psi|^2}{\psi}\leq 10n(R-r)^{-2}\chi_{B_R\setminus B_r}\,,
\eann
where $\chi_{B_R\setminus B_r}$ denotes the characteristic function of the set ${B_R\setminus B_r}$ (pulled back by the flow parametrization $X$). %It will often prove convenient to conflate a set $W\subset \R^{n+1}\times \R$ with its preimage under $\mathcal{X}$.

We thus obtain
\ba\label{eq:evolvepsiGes}
\frac{(\pd_t-\Delta)\psi^2v^2}{\psi^2v^2}\leq{}&\frac{100n}{(R-r)^2}\frac{\chi_{B_R\setminus B_r}}{\psi}-2\lb 1-\frac{2+2\gamma^{-1}}{p}\rb\frac{|\cd v|^2}{v^2}\nonumber\\
{}&-\gamma p\frac{|\cd A|^2}{H^2}+\sigma p|A|^2\nonumber\\
\leq{}&\frac{100n}{(R-r)^2}\frac{\chi_{B_R\setminus B_r}}{\psi}-\frac{4}{3}\frac{|\cd v|^2}{v^2}-\gamma p\frac{|\cd A|^2}{H^2}+\sigma p|A|^2
\ea
wherever $\psi v>0$.

\subsection*{The $L^2$-estimate}\label{sec:L2est}

Applying \eqref{eq:evolvepsiGes} and the gradient flow property of mean curvature flow, we obtain
\ba\label{eq:basicesimate}
\frac{d}{dt}\!\int\!\psi^2v^2d\mu+\!\int\!\psi^2v^2H^2\,d\mu={}&\int\pd_t(\psi^2v^2)d\mu\nonumber\\
\leq{}&\!\int\!\psi^2v^2\!\lb \sigma p|A|^2-\frac{4}{3}\frac{|\cd v|^2}{v^2}-\gamma p\frac{|\cd A|^2}{H^2}\rb\!d\mu\nonumber\\
{}&+\frac{100n}{(R-r)^2}\int_{B_R\setminus B_r}\psi v^2\,d\mu\,.
\ea

\begin{comment}
Recall the following Poincar\'e-type inequality for functions supported away from cylindrical points \cite[Proposition 2.7]{L17}.

\begin{prop}[Poincar\'e-type inequality]\label{prop:Poincareinequality}
Let $M^n\to \R^{n+1}$ be a hypersurface of dimension $n\geq 2$ and let $u\in W^{2,1}(M)$ be a compactly supported function satisfying
\[
\mathrm{spt}\,u\Subset \{x\in M:A_x\in \Gamma\}
\]
for some cone $\Gamma\Subset \{W\in\mathrm{Sym}^{n\times n}:\tr(W)>0\}\setminus \mathrm{Cyl}$, where $\mathrm{Cyl}$ denotes the set of cylindrical points in $\mathrm{Sym}^{n\times n}$.

For any $\beta>0$,
\begin{align*}
\int u^2\vert A\vert^2\,d\mu%\leq{}&P\int u^2\left(\frac{\vert\nabla {\operatorname{II}}\vert^2}{H^2}+\frac{\vert\nabla {\operatorname{II}}\vert}{H}\frac{\vert\nabla u\vert}{u}\right)\,d\mu\\
%\leq{}&P\int u^2\left(\frac{\vert\nabla {\operatorname{II}}\vert^2}{H^2}+\frac{\delta}{P}\frac{\vert\nabla u\vert^2}{u^2}+\frac{P}{4\delta}\frac{\vert\nabla {\operatorname{II}}\vert^2}{H^2}\right)\,d\mu\\
\leq{}&\beta \int \vert\cd u\vert^2\,d\mu+P(1+\beta^{-1})\int u^2\frac{\vert\nabla A\vert^2}{H^2}\,d\mu\,,
\end{align*}
where $P$ depends only on $n$ and $\Gamma_0$.
\end{prop}
Applied to $u=\psi v$ this yields, for any $\beta>0$,
\end{comment}
We now apply the Poincar\'e-type inequality \cite[Proposition 2.7]{L17} to $u=\psi v$. This yields, for any $\beta>0$,
\bann
\int \psi^2v^2|A|^2d\mu\leq{}& \beta\int \psi^2v^2\left\vert \frac{\cd\psi}{\psi}+\frac{\cd v}{v}\right\vert^2d\mu+P(1+\beta^{-1})\int \psi^2v^2\frac{|\cd A|^2}{H^2}d\mu\\
\leq{}&2\beta\int \left(\psi^2|\cd v|^2+v^2|\cd\psi|^2\right)\,d\mu\\
{}&+P(1+\beta^{-1})\int \psi^2v^2\frac{|\cd A|^2}{H^2}\,d\mu\,,
\eann
where $P=P(n,\alpha,\varepsilon)$. Setting $\beta=p^{-\frac{1}{2}}$ and recalling \eqref{eq:basicesimate}, we find that
\bann
\frac{d}{dt}\int\psi^2v^2\,d\mu\leq{}&\frac{(100+20\sigma p^{\frac{1}{2}})n}{(R-r)^2}\int_{B_R\setminus B_r}\psi\,v^2\,d\mu\\
{}&+\lb 2\sigma p^{\frac{1}{2}}-\frac{4}{3}\rb \int \psi^2|\cd v|^2\,d\mu\\
{}&+\lb\sigma P(1+p^{\frac{1}{2}})-\gamma\rb p\int\psi^2v^2\frac{|\cd A|^2}{H^2}\,d\mu\,.
\eann
Choosing $p\geq L$ and $\sigma\leq \ell p^{-\frac{1}{2}}$, where $\ell=\ell(n,\alpha,\varepsilon)\le 4$ and $L=L(n,\alpha,\varepsilon)$, yields
\ba\label{eq:Lpmonotone}
\frac{d}{dt}\int\psi^2v^2\,d\mu\leq{}&\frac{200n}{(R-r)^2}\int_{B_R\setminus B_r}\psi v^2\,d\mu\,.
\ea
If we choose $R=\lambda$ and $r=\lambda-\delta$, then, integrating in time, we find that
\bann
\sup_{t\in(0,\frac{1}{2n}R^2)}\int_{B_{\lambda-\delta}}v^2\,d\mu\leq{}&\int_{B_{\lambda}}v^2\,d\mu_0+200n\delta^{-2}\int\!\!\!\int_{B_\lambda\setminus B_{\lambda-\delta}}v^2\,d\mu_t\,dt\,,
\eann
and hence
\begin{equation}\label{eq:Lpest}
\int\!\!\!\int_{B_{\lambda-\delta}}v^2\,d\mu\,dt\le \frac{1}{2n}\int_{B_{\lambda}}v^2\,d\mu_0+100\delta^{-2}\int\!\!\!\int_{B_\lambda\setminus B_{\lambda-\delta}}v^2\,d\mu_t\,dt\,,%\nonumber\\
\end{equation}
%where $C$ depends on $n$, $\alpha_0$ and $p$. Setting $R=1$ and $r=R_0\doteqdot \vartheta$, our hypotheses yield
%\ba\label{eq:Lpest}
%\int_{B_{R_0}}v^2\,d\mu\leq{}&C(n,\eta,\vartheta,\alpha,\sigma,p)%\left.\int_{B_{R_0}}\Gesp^p\,d\mu\right|_{t=-T}+100\vartheta^{-2}V\Theta^{\sigma p}\,.
%\ea
so long as $p\geq L$ and $\sigma\leq \ell p^{-\frac{1}{2}}$.

%This provides an estimate in $L^2$ for the function $v\doteqdot \Gesp^\frac{p}{2}$, so long as $p\gg 1$ and $\sigma p^{\frac{1}{2}}\ll 1$.

\subsection*{From $L^2$ to $L^\infty$}

Stampacchia iteration will now allow us to pass from $L^2$ to $L^\infty$. We assume the reader is familiar with \cite[Section 5]{Stampacchia66} or \cite[Chapter II: Appendices B and C]{StKi}.

Given $k\ge k_0\doteqdot \Theta^\sigma\ge \sup_{B_\lambda\setminus B_{\lambda-\delta}}G_{\varepsilon,\sigma}$ and $R\le \lambda-\delta$, consider
\bann
v_k^2\doteqdot \lb\Ges-k\rb^p_+\quad\text{and}\quad A_{k,R}\doteqdot \{(x,t)\in X^{-1}(B_R):v_k(x,t)>0\}
\eann
and set
\bann
u(k,R)\doteqdot \int\!\!\!\int_{A_{k,R}} v_k^2\,d\mu_t\,dt\quad\text{and}\quad a(k,R)\doteqdot \int\!\!\!\int_{A_{k,R}}d\mu_t\,dt\,.
\eann
%where, given $W\subset \R^{n+1}\times \R$,
%\[
%\int_Wd\mu_t\doteqdot \int_{X_t^{-1}(W)}d\mu_t\,.
%\]
Note that, for any $h\geq k>0$ and any $0<r\leq R\le \lambda-\delta$,
%\bann
%a(h,r)%=\int_{-\eta^2}^0\hspace{-1mm}\int_{A_{h,r}}\frac{(\Ges-k)_+^p}{(\Ges-k)_+^p}
%\leq\int_{-\eta^2}^0\hspace{-1mm}\int_{A_{h,r}}\frac{(\Ges-k)_+^p}{(h-k)^p}\,d\mu\,dt
%\eann
%so that
\ba\label{eq:iteration1}
(h-k)^pa(h,r)\leq u(k,r)\,.
\ea

We need an estimate for $u(k,r)$. First observe that, computing as in \eqref{eq:basicesimate}, we can estimate
\bann%\label{eq:evolvevk}
\frac{d}{dt}\int\psi^2v_k^2\,d\mu+\int_{A_{k,r}}\hspace{-3mm}|\cd v_k|^2\,d\mu+\int_{A_{k,r}}\hspace{-3mm}v_k^2H^2\,d\mu
\leq{}& \frac{100n}{(R-r)^2}\int_{A_{k,R}}\hspace{-3mm}v_k^2\,d\mu\\
{}&+\sigma p\int_{A_{k,R}}\hspace{-3mm}\Ges^p|A|^2\,d\mu
\eann
for any $k>0$ and $r<R\le \lambda-\delta$, where $\psi$ is a cut-off function satisfying $\psi\equiv 1$ on $B_r$ and $\psi\equiv 0$ outside of $B_R$. 
%\begin{rem}
%In fact, we can actually get
%\bann
%\frac{d}{dt}\int\psi^2v_k^2+\int\psi^2|\cd v_k|^2\leq{}& \frac{100}{(R-r)^2}\int_{B_R\setminus B_r}\psi v_k^2+\sigma k^p\int_{A_{k,R}}\psi^2|A|^2\,.
%\eann
%\end{rem}
%
%
%Michael--Simon Sobolev
%\begin{comment}
On the other hand, the Sobolev inequality of Michael and Simon \cite[Theorem 2.1]{MS73} %(as stated in \cite[Theorem 1]{Hutch}, for example) 
yields%\bann
%\frac{2}{c_S}\lb\int_{A_{k,r}}\hspace{-3mm}v_k^{2^\ast}\,d\mu\rb^\frac{2}{2^\ast}\leq \int_{A_{k,r}}\hspace{-3mm}|\cd v_k|^2\,d\mu+\lb\int_{A_{k,r}}\hspace{-3mm}H^n\,d\mu\rb^{\frac{2}{n}}\lb\int_{A_{k,r}}\hspace{-3mm}v_k^{2^\ast}\,d\mu\rb^\frac{2}{2^\ast}\,,
%\eann
\bann
\lb\int_{A_{k,r}}v_k^{2^\ast}\,d\mu\rb^\frac{2}{2^\ast}\leq c_S\int_{A_{k,r}}\lb|\cd v_k|^2+H^2v_k^2\rb\,d\mu\,,
\eann
where\footnote{We can interpret the left hand side as the square of the $L^\infty$-norm when $n=2$ with $2^\ast$ any fixed number bigger than one and the constant $c_S$ depending additionally on $2^\ast$ and the measure of the support of $v_k$. Cf. \cite{Hu84}.}, for $n\ge 3$, $\frac{1}{2^\ast}=\frac{1}{2}-\frac{1}{n}$ and $c_S$ depends only on $n$, %Of course, this is useless unless the integral of $H^n$ is small. Noting that $\sigma'\doteqdot \sigma+\frac{n}{p}\sim p^{-\frac{1}{2}}$ for $p$ large, we can estimate, using \eqref{eq:Lpest} (choosing $\ell$ slightly smaller), 
%\bann
%\int_{A_{k,r}}\hspace{-3mm}H^n\,d\mu\leq{}& k^{-p}\int_{A_{k,r}}\hspace{-3mm}H^n\Ges^p\,d\mu\leq k^{-p}\int_{B_{R_0}}\hspace{-2mm}G_{\varepsilon,\sigma',+}^p\,d\mu%\\\leq{}& 100\vartheta^{-2}k^{-p}\int_{-\eta^2}^0\hspace{-1mm}\int_{B_1\setminus B_{R_0}}G_{\varepsilon,\sigma',+}^p
%\leq \frac{C(n,V,\Theta,\eta,\vartheta,\sigma,p)}{k^p}\,.%\frac{100V\Theta^{\sigma' p}}{\vartheta^{2}k^{p}}\,.
%\eann
%Thus, for $k\geq k_0=k_0(n,V,\Theta,\eta,\vartheta,\sigma,p)$, %setting $k^p_0\doteqdot 100\cdot (2c_S)^{\frac{n}{2}}\vartheta^{-2}V\Theta^{\sigma' p}$, 
%we obtain
%\bann
%\lb\int_{A_{k,r}}\hspace{-3mm}v_k^{2^\ast}\,d\mu\rb^\frac{2}{2^\ast}\leq{}& c_S\int_{A_{k,r}}\hspace{-3mm}|\cd v_k|^2\,d\mu
%\lb\int_{A_{k,r}}(\psi v_k)^{2^\ast}\rb^\frac{2}{2^\ast}\leq{}& 2c_S\int_{A_{k,r}}|\cd (\psi v_k)|^2%\\
%\leq{}& 4c_S\int_{A_{k,r}}\lb\psi^2|\cd v_k|^2+v_k^2|\cd\psi|^2\rbHu
%\eann
%\end{comment}
\begin{comment}
On the other hand, recalling the Sobolev inequality of Hutchinson\footnote{We can interpret the left hand side as the square of the $L^\infty$-norm when $n=2$.}, we have
\bann
\lb \int_{A_{k,r}}\hspace{-3mm}v_k^{2^\ast}d\mu\rb^{\frac{2}{2^\ast}}\leq c_S\int_{A_{k,r}}\lb |\cd v_k|^2+H^2v_k^2\rb\,,
\eann
where $\frac{1}{2^\ast}=\frac{1}{2}-\frac{1}{n}$ and $c_S>1$ depends only on $n$, 
\end{comment}
so that
\bann
\frac{d}{dt}\int\psi^2v_k^2\,d\mu+\frac{1}{c_S}\lb\int_{A_{k,r}}v_k^{2^\ast}\,d\mu\rb^\frac{2}{2^\ast}\leq{}& \frac{100n}{(R-r)^2}\int_{A_{k,R}}v_k^2\,d\mu\\
{}&+\sigma p\int_{A_{k,R}}\Ges^p|A|^2\,d\mu\,.
\eann
Integrating with respect to $t$ then yields
\bann
\sup_{t\in(0,\frac{1}{2n}R^2)}\int_{A_{k,r}}\hspace{-3mm}v_k^2\,d\mu+&\int\hspace{-1mm}\lb\int_{A_{k,r}}\hspace{-3mm}v_k^{2^\ast}\,d\mu\rb^\frac{2}{2^\ast}\hspace{-1mm}dt\\
\leq{}&\frac{100nc_S}{(R-r)^2}\int\!\!\!\int_{A_{k,R}}v_k^2\,d\mu\,dt+c_S\sigma p\int\!\!\!\int_{A_{k,R}}\Ges^p|A|^2\,d\mu\,dt\,.
\eann
By the interpolation inequality,
\bann
\int_{A_{k,r}}v_k^{\frac{2(n+2)}{n}}\,d\mu\leq \lb\int_{A_{k,r}}v_k^2\,d\mu\rb^{\frac{2}{n}}\lb\int_{A_{k,r}}v_k^{2^\ast}\,d\mu\rb^{\frac{2}{2^\ast}}
\eann
and hence, by Young's inequality,
\bann
\lb\int\!\!\!\int_{A_{k,r}}\hspace{-3mm}v_k^{\frac{2(n+2)}{n}}d\mu\,dt\rb^{\frac{n}{n+2}}\leq{}& \lb\sup_{t\in(0,\frac{1}{2n}R^2)}\int_{A_{k,r}}\hspace{-3mm}v_k^2\,d\mu\rb^{\frac{2}{n+2}}\lb\int\hspace{-1mm}\lb\int_{A_{k,r}}\hspace{-3mm}v_k^{2^\ast}d\mu\rb^{\frac{2}{2^\ast}}\hspace{-1mm}dt\rb^{\frac{n}{n+2}}\\
\leq{}& \tfrac{2}{n+2}\sup_{t\in(0,\frac{1}{2n}R^2)}\int_{A_{k,r}}\hspace{-3mm}v_k^2\,d\mu+\tfrac{n}{n+2}\int\hspace{-1mm}\lb\int_{A_{k,r}}\hspace{-3mm}v_k^{2^\ast}d\mu\rb^{\frac{2}{2^\ast}}\hspace{-1mm}dt\,.
\eann
Thus,
\ba\label{eq:uest1}
\lb\int\!\!\!\int_{A_{k,r}}v_k^{\frac{2(n+2)}{n}}d\mu\,dt\rb^{\frac{n}{n+2}}\leq{}&
\frac{100nc_S}{(R-r)^2}\int\!\!\!\int_{A_{k,R}}v_k^2\,d\mu\,dt\nonumber\\
{}&+c_S\sigma p\int\!\!\!\int_{A_{k,R}}\Ges^p|A|^2\,d\mu\,dt\,.
\ea
Applying H\"older's inequality and (choosing $\ell$ slightly smaller\footnote{Depending now also on $\rho$, which will be fixed momentarily.}), we estimate, for $\sigma'\doteqdot \sigma+\frac{2}{p}$ and some soon-to-be-determined $\rho\geq 1$,
\ba\label{eq:uest2}
\int\!\!\!\int_{A_{k,R}}H^2\Ges^p\,d\mu\,dt={}&\int\!\!\!\int_{A_{k,R}}G_{\varepsilon,\sigma',+}^p\,d\mu\,dt\nonumber\\
%\leq{}&a(k,R)^{1-\frac{1}{\rho}}\lb\int\!\!\int_{A_{k,R}}H^{2\rho}\Ges^{p\rho}\,d\mu\,dt\rb^{\frac{1}{\rho}}\nonumber\\
\le{}& a(k,R)^{1-\frac{1}{\rho}}\lb\int\!\!\!\int_{A_{k,R}}G_{\varepsilon,\sigma',+}^{p\rho}\,d\mu\,dt\rb^{\frac{1}{\rho}}\nonumber\\
\leq{}& a(k,R)^{1-\frac{1}{\rho}}\lb\int\!\!\!\int_{B_{\lambda-\delta}}G_{\varepsilon,\sigma',+}^{p\rho}\,d\mu\,dt\rb^{\frac{1}{\rho}}\,.%\nonumber\\
%\leq{}& %(100\vartheta^{-2}V)^{\frac{1}{r}}\Theta^{\sigma'p}
%\big(c(n)\delta^{-2}V\big)^{\frac{1}{\rho}}\Theta^{\sigma' p}a(k,R)^{1-\frac{1}{\rho}}
\ea

Similarly, we may estimate, for any $k>k_0$ and $R<\lambda-\delta$,
\ba\label{eq:iteration3}
u(k,R)%\doteqdot {}&\int\hspace{-1mm}\int_{A_{k,R}}v_k^2\nonumber\\
%\leq{}& a(k,R)^{1-\frac{1}{\rho}}\lb\int\!\!\!\int_{A_{k,R}}\Gesp^{p\rho}\,d\mu\,dt\rb^{\frac{1}{\rho}}\nonumber\\
\leq{}& a(k,R)^{1-\frac{1}{\rho}}\lb\int\!\!\!\int_{B_{\lambda-\delta}}\Gesp^{p\rho}\,d\mu\,dt\rb^{\frac{1}{\rho}}\,.%\nonumber\\
%\leq{}&%(100\vartheta^{-2}V)^{\frac{1}{r}}\Theta^{\sigma p}
%\big(c(n)\delta^{-2}V\big)^{\frac{1}{\rho}}\Theta^{\sigma p}a(k,R)^{1-\frac{1}{\rho}}\,.
\ea

Finally, we estimate
\ba\label{eq:uest3}
u(k,r)%\doteqdot {}&\int\hspace{-1mm}\int_{A_{k,r}}v_k^2
\leq a(k,r)^{\frac{2}{n+2}}\lb\int\!\!\!\int_{A_{k,r}}v_k^{\frac{2(n+2)}{n}}d\mu\,dt\rb^{\frac{n}{n+2}}\,.
\ea
Since $\vert A\vert\le \alpha^{-1}H$ and $G\le H$, combining \eqref{eq:iteration1}--%, \eqref{eq:uest1}, \eqref{eq:uest2}, \eqref{eq:iteration3} and 
\eqref{eq:uest3} and the $L^2$-estimate \eqref{eq:Lpest} yields
%\ba\label{eq:iteration2}
%(R-r)^2u(k,r)\leq{}&Ca(k,r)^{\frac{2}{n+2}}\lb u(k,R)+(R-r)^2a(k,R)^{1-\frac{1}{\rho}}\rb\nonumber\\
%\leq{}&Ca(k,r)^{\frac{2}{n+2}}\lb u(k,R)+a(k,R)^{1-\frac{1}{\rho}}\rb
%=:{}&C\frac{a(h,r)^{\frac{2}{n+2}}}{(R-r)^2}\lb u(k,R)+L a(k,R)^{1-\frac{1}{r}}\rb\,.\nonumber
%\leq{}&2C\frac{a(h,r)^{\frac{2}{n+2}}}{(R-r)^2}\max\lcb u(k,R),L a(k,R)^{1-\frac{1}{r}}\rcb\,.
%\ea
%for any $k>k_0$ and $r<R<R_0$, where $C=C(n,\alpha,V,\Theta,\delta,\sigma,p,\rho)$. %That is, setting $L\doteqdot \sigma p \vartheta^{-\frac{2}{r}}\sup_{Q_1\setminus Q_R}H^{p\sigma'}$,
%\ba\label{eq:iteration2}
%(R-r)^2u(h,r)\leq Ca(h,r)^{\frac{2}{n+2}}\lb u(k,R)+L a(k,R)^{1-\frac{1}{r}}\rb\,.
%\ea
%
%
%Combining \eqref{eq:iteration1}, \eqref{eq:iteration2} and \eqref{eq:iteration3}, we conclude that
\bann
(h-k)^p(R-r)^2a(h,r)%\leq{}& (R-r)^2u(k,r)\\
%\leq{}&Ca(k,r)^{\frac{2}{n+2}}\lb a(k,R)^{1-\frac{1}{r}}+u(k,R)\rb\\
%\leq{}&Ca(k,R)^{\gamma}\lb 1+\frac{u(k,R)}{a(k,R)^{1-\frac{1}{r}}}\rb\\
%\le{}&(R-r)^2u(k,r)\\
%\le{}&(R-r)^2a(k,r)^{\frac{2}{n+2}}\lb\int\!\!\!\int_{A_{k,r}}v_k^{\frac{2(n+2)}{n}}d\mu\,dt\rb^{\frac{n}{n+2}}\\
%\le{}&ca(k,r)^{\frac{2}{n+2}}\lb u(k,R)+\sigma p(R-r)^2\int\!\!\!\int_{A_{k,R}}\Ges^pH^2\,d\mu\,dt \rb\\
%\leq{}&ca(k,R)^{\gamma}\!\lsb\!\lb\!\int\!\!\!\int_{U_{R_0}}\!\!\!G_{\varepsilon,\sigma,+}^{p\rho}d\mu\,dt\rb^{\!\!\frac{1}{\rho}}\!\!+\sigma p(R-r)^2\!\lb\!\int\!\!\!\int_{U_{R_0}}\!\!\!G_{\varepsilon,\sigma',+}^{p\rho}d\mu\,dt\rb^{\!\!\frac{1}{\rho}}\rsb\\
\leq{}&ca(k,R)^{\gamma}\big(1+\sigma p\Theta^2(R-r)^2\big)\cdot\\
{}&\!\lb\!\frac{1}{2n}\!\int_{B_{\lambda}}\!\!G_{\varepsilon,\sigma,+}^{p\rho}d\mu_0+\frac{100}{\delta^2}\!\int\!\!\!\int_{B_\lambda\setminus B_{\lambda-\delta}}\!\!\!\!\!\!\!\!G_{\varepsilon,\sigma,+}^{p\rho}d\mu\,dt\rb^{\frac{1}{\rho}}\\
\leq{}&ca(k,R)^{\gamma}\big(1+\sigma p\Theta^2(R-r)^2\big)\cdot\\
{}&\lb\!\frac{1}{2n}\!\int_{B_{\lambda}}\!\!H^{\sigma p \rho}d\mu_0+\frac{100}{\delta^2}\!\int\!\!\!\int_{B_\lambda\setminus B_{\lambda-\delta}}\!\!\!\!\!\!\!\!H^{\sigma p\rho}d\mu\,dt\rb^{\frac{1}{\rho}}\\
\leq{}&ca(k,R)^{\gamma}\big(1+\sigma p\Theta^2(\lambda-\delta)^2\big)\Theta^{\sigma p(1-\frac{1}{\rho})}\Lambda^{\frac{\sigma p}{\rho}}
%\le {}& Ca(k,R)^{\gamma}\,,
\eann
so long as $p\ge L(n,\alpha,\varepsilon,\rho)$ and $\sigma\le \ell(n,\alpha,\varepsilon,\rho)p^{-\frac{1}{2}}$, where $c=c(n,\alpha,\rho)$, $\gamma\doteqdot 1+\frac{2}{n+2}-\frac{1}{\rho}$, and
%\bann
%C\doteqdot{}& c\big(\Theta^2+\sigma p(R-r)^2\big)\lb\frac{100}{\delta^2}\int\!\!\!\int_{C_\delta(U)}G_{\varepsilon,\sigma,+}^{p\rho}\,d\mu\,dt\rb^{\frac{1}{\rho}}\\
%\le{}& c\big(\Theta^2+\sigma p(R-r)^2\big)\lb\frac{100}{\delta^2}\int\!\!\!\int_{C_\delta(U)}H^{\sigma p\rho}\,d\mu\,dt\rb^{\frac{1}{\rho}}\\
%\le{}& c\big(\Theta^2+\sigma p(R-r)^2\big)\lb\frac{100}{\delta^2}\rb^{\frac{1}{\rho}}\Theta^{\sigma p(1-\frac{1}{\rho})}\Lambda^{\frac{\sigma p}{\rho}}\,,
%=c(n,\alpha,\varepsilon,\rho,\sigma,p)(\delta^{-2}V)^{\frac{1}{\rho}}\Theta^{\sigma p+2}.
%\eann
%where
\[
\Lambda\doteqdot\lb\frac{1}{2n}\int_{B_{\lambda}}H^{\sigma p}\,d\mu_0+\frac{100}{\delta^2}\int\!\!\!\int_{B_\lambda\setminus B_{\lambda-\delta}}H^{\sigma p}\,d\mu\,dt\rb^{\frac{1}{\sigma p}}\,.
\]

%\begin{proof}[Proof of Theorem \ref{thm:localpinching}]
At this point, we fix some $\rho=\rho(n)>1+\frac{n}{2}$ (so that $\gamma=\gamma(n)>1$) and choose $p=p(n,\alpha,\varepsilon,q)\ge L$ and $\sigma=\sigma(n,\alpha,\varepsilon,q)\le \ell p^{-\frac{1}{2}}$ such that $\sigma p=q$. Stampacchia's Lemma \cite[Lemma 5.1]{Stampacchia66} then yields
\bann
a\lb k_0+d,\vartheta R_0\rb=0\,,
\eann
where $R_0\doteqdot \lambda-\delta$ and
\bann
d^p\doteqdot %\frac{2^{(p+2)\gamma/(\gamma-1)}Ca(k_0,R_0)^{\gamma-1}}{(1-\vartheta)^2R_0^2}\,.
\frac{2^{\frac{(p+2)\gamma}{\gamma-1}}c\big(1+\sigma p\Theta^2(\lambda-\delta)^2\big)\Theta^{\sigma p(1-\frac{1}{\rho})}\Lambda^{\frac{\sigma p}{\rho}}a(k_0,R_0)^{\gamma-1}}{(1-\vartheta)^2(\lambda-\delta)^2}\,.
\eann
We may estimate, using the $L^2$-estimate \eqref{eq:Lpest},
\bann
a(k_0,R_0)\leq{}& k_0^{-p}\int\!\!\!\int_{U_{R_0}}\Gesp^p\,d\mu\,dt\\
\leq{}& k_0^{-p}\lb \frac{1}{2n}\int_{B_{\lambda}}\Gesp^p\,d\mu_0+\frac{100}{\delta^2}\int\!\!\!\int_{B_\lambda\setminus B_{\lambda-\delta}}\Gesp^p\,d\mu\,dt\rb\\
\leq{}& k_0^{-p}\Lambda^{\sigma p}\,.
%\leq 100k_0^{-p}\delta^{-2}\Theta^{\sigma p}V\,,%\frac{100V\Theta^{\sigma p}}{k_0^p}
\eann
%where
%\[
%\Lambda_2\doteqdot \left(\int\!\!\!\int_{C_\delta(U)\cap\{d_m>\varepsilon H\}}H^{\sigma p}\,d\mu_t\,dt\right)^\frac{1}{\sigma p}\,,
%\]
Since we chose $k_0=\Theta^\sigma$ and $\sigma p=q$, we thus obtain
\bann
\Ges\leq{}& k_0+d\\
={}&k_0\left(1+2^{\frac{(p+2)\gamma}{p(\gamma-1)}}\lsb\frac{c\big(1+\sigma pR_0^2\Theta^2\big)}{(1-\vartheta)^2R_0^2}\rsb^{\frac{1}{p}}\frac{\Theta^{\sigma(1-\frac{1}{\rho})}}{k_0^{1-\frac{1}{\rho}}}\frac{\Lambda^{\frac{2}{n+2}\sigma}}{k_0^{\frac{2}{n+2}}}\right)\\
\le{}&%\Theta^\sigma\left(1+2^{\frac{(p+2)\gamma}{p(\gamma-1)}}\lsb\frac{c\big(1+q R_0^2\Theta^2\big)}{(1-\vartheta)^2R_0^2}V^2\rsb^{\frac{1}{p}}\right)\\
%={}&
\Theta^\sigma\left(1+c(n,\alpha,q,\varepsilon)\lsb\frac{(\lambda-\delta)^{-2}+q\,\Theta^2}{(1-\vartheta)^2}V^2\rsb^{\frac{\sigma}{q}}
\right)
\eann
in $B_{\vartheta(\lambda-\delta)}$. Young's inequality then yields %(under very mild assumptions on the allowed values of the constants\footnote{Namely, $\sqrt{q}\,\Theta\ge (\lambda-\delta)^{-1}$ and $\Theta V\ge 1-\vartheta$.})
\bann
G%\le{}& 2\varepsilon H+c(n,\alpha,\sigma,p,\varepsilon)\sigma\left(\frac{1-\sigma}{\varepsilon}\right)^\frac{1-\sigma}{\sigma}\Theta\left[\frac{1+\sigma p\Theta^2(\lambda-\delta)^2}{(1-\vartheta)^2(\lambda-\delta)^2}\right]^{\frac{1}{\sigma p}}\lb\frac{\Lambda}{\Theta}\rb^{\frac{2}{n+2}}\\%\frac{\Lambda_1}{\Theta}\frac{\Lambda_2^{\gamma-1}}{\Theta^{\gamma-1}}\\
%={}& 2\varepsilon H+c(n,\alpha,q,\varepsilon)\Theta\left(\frac{(\lambda-\delta)^{-1}+\sqrt{q}\,\Theta}{1-\vartheta}\right)^{\frac{2}{q}}\lb\frac{\Lambda}{\Theta}\rb^{\frac{2}{n+2}}\\
%\le{}& 2\varepsilon H+c(n,\alpha,\sigma,p,\varepsilon)\sigma\left(\frac{1-\sigma}{\varepsilon}\right)^\frac{1-\sigma}{\sigma}\Theta\left[\frac{\Theta^2V^{2}}{(1-\vartheta)^2\delta^{\frac{4}{n+2}}}\right]^{\frac{1}{\sigma p}}\\
\le{}& 2\varepsilon H+c(n,\alpha,q,\varepsilon)\Theta \left(\frac{\Theta V}{1-\vartheta}\right)^{\frac{2}{q}}
\eann
in $B_{\vartheta(\lambda-\delta)}$. This completes the proof of Theorem \ref{thm:convexity_and_cylindrical_estimates2}. %(and justifies Remark \ref{rem:weaker hypotheses}).
\end{proof}

%\section{The derivative estimates}

%\input{gradient_estimate}

%\printbibliography
\bibliographystyle{plain}
\bibliography{../bibliography}

\end{document}